\documentclass[12pt]{amsart}
\usepackage{amsthm,amsmath,amssymb,amscd,graphicx,enumerate, stmaryrd,xspace,verbatim, epic, eepic,url,fullpage}

\usepackage[usenames,dvipsnames]{color}

\usepackage[colorlinks=true]{hyperref}

\usepackage[active]{srcltx}
\usepackage[all]{xypic}
\SelectTips{cm}{}

{
   \newtheorem{theorem}[subsubsection]{Theorem}
      \newtheorem*{theorem*}{Theorem}
   \newtheorem{proposition}[subsubsection]{Proposition}

   \newtheorem{lemma}[subsubsection]{Lemma}

   \newtheorem*{conjecture*}{Conjecture}

}
{\theoremstyle{definition}
          \newtheorem*{exercise*}{Exercise}

   \newtheorem*{example*}{Example}
   \newtheorem{definition}[subsubsection]{Definition}
   
   \newtheorem*{definition*}{Definition}
   
   \newtheorem{remark}[subsubsection]{Remark}

}
\newcommand{\QQ}{{\mathbb{Q}}}

\newcommand{\ZZ}{{\mathbb{Z}}}

\newcommand{\GG}{{\mathbb{G}}}

\renewcommand{\AA}{{\mathbb{A}}}

\newcommand{\cA}{{\mathcal A}}

\renewcommand{\cD}{{\mathcal D}}

\newcommand{\cI}{{\mathcal I}}
\newcommand{\cJ}{{\mathcal J}}

\newcommand{\cM}{{\mathcal M}}

\newcommand{\cO}{{\mathcal O}}

\newcommand{\cQ}{{\mathcal Q}}

\newcommand{\oJ}{{\overline{J}}}

\def\<{\langle}
\def\>{\rangle}

\newcommand{\Spec}{\operatorname{Spec}}

\newcommand{\maxord}{{\operatorname{max-ord}}}
\newcommand{\maxlogord}{{\operatorname{maxlogord}}}

\def\:{{\colon}}
\def\.{{,\dots,}}

\def\inv{{\rm inv}}

\def\maxinv{{\rm maxinv}}
\def\maxloginv{{\rm maxloginv}}
\def\loginv{{\rm loginv}}

\newcommand{\double}{\genfrac..{0pt}1
{\raise -1pt\hbox{$\scriptstyle\longrightarrow$}}{\raise 3pt\hbox
{$\scriptstyle\longrightarrow$}}}

\renewcommand{\setminus}{\smallsetminus}

\def\int{{\rm int}}

\def\tototi{\mathbin{\mathop{\otimes}\limits^{\raise-1pt\hbox
{$\scriptscriptstyle {\rm L}$}}}}

\def\indlim{\mathop{\vrule width0pt height7pt depth
4pt\smash{\lim\limits_{\raise 1pt\hbox to 14.5pt
{\rightarrowfill}}}}}
\def\projlim{\mathop{\vrule width0pt height7pt depth
4pt\smash{\lim\limits_{\raise 1pt\hbox to 14.5pt
{\leftarrowfill}}}}}

\newcommand\displaceamount{3pt}

\newcommand{\doubledown}{\ar@<\displaceamount>[d]\ar@<-\displaceamount>[d]}

\newcommand{\doubleup}{\ar@<\displaceamount>[u]\ar@<-\displaceamount>[u]}

\newcommand{\doubleright}{\ar@<\displaceamount>[r]\ar@<-\displaceamount>[r]}

\newcommand{\logord}{{\operatorname{logord}}}

\newcommand{\ord}{{\operatorname{ord}}}

\begin{document}
\title[Log resolution using weighted blow-ups]{Logarithmic resolution of singularities in characteristic 0 using weighted blow-ups}

\author[Abramovich]{Dan Abramovich}
\address{Department of Mathematics, Box 1917, Brown University,
Providence, RI, 02912, U.S.A}
\email{dan\_abramovich@brown.edu}

\author[Belotto da Silva]{Andr\'e Belotto da Silva}
\address{Universit\'e Paris Cit\'e, UFR de Math\'ematiques,
Institut de Math\'ematiques de Jussieu-Paris Rive Gauche,
B\^atiment Sophie Germain, Bureau 716,
75205 Paris Cedex 13, France}
\email{belotto@imj-prg.fr}

\author[Quek]{Ming Hao Quek}
\address{Harvard University, Department of Mathematics, 1 Oxford Street, Cambridge, MA 02138}
\email{mhquek@math.harvard.edu}

\author[Temkin]{Michael Temkin}
\address{Einstein Institute of Mathematics\\
               The Hebrew University of Jerusalem\\
                Edmond J. Safra Campus, Giv'at Ram, Jerusalem, 91904, Israel}
\email{temkin@math.huji.ac.il}

\author[W{\l}odarczyk] {Jaros{\l}aw W{\l}odarczyk}
\address{Department of Mathematics, Purdue University\\
150 N. University Street,\\ West Lafayette, IN 47907-2067}
\email{wlodar@math.purdue.edu}


\thanks{This research is supported by BSF grants 2018193 and 2022230, ERC Consolidator Grant 770922 - BirNonArchGeom, NSF grants 
DMS-2100548, DMS-2401358,   the Plan d'Investissements France 2030, IDEX UP ANR-18-IDEX-0001, and Simon foundation grants MPS-SFM-00006274, MPS-TSM-00008103. Dan Abramovich would like to thank IHES and the Hebrew University of Jerusalem for welcoming environments during periods of preparation of this article.}

\dedicatory{To James McKernan}

\date{\today}

\begin{abstract}
In characteristic zero, we construct logarithmic resolution of singularities, with simple normal crossings exceptional divisor, using weighted blow-ups.
\end{abstract}
\maketitle

\section{Introduction}

\subsection{Resolution of singularities in characteristic 0}

A thread running through James McKernan's work is that 
\begin{quote} 
every  result in mathematics deserves a natural, fully motivated and understandable proof.
\end{quote}
 The more important and difficult the result, the more urgent the need for such proofs. A proof adhering to this will be said to \emph{satisfy McKernan's criterion}.

One such important and difficult result is Hironaka's theorem on resolution of singularities in characteristic zero:

\begin{theorem}[Hironaka, 1964] Any variety $X$ over a field of characteristic zero admits a resolution of singularities $X' \to X$.\end{theorem}
Of course, more is true: 
\begin{enumerate} 
	\item Hironaka's resolution is \emph{embedded}, in other words, if $X\subset Y$ is an embedding in a nonsingular variety, then the resolution $X' \to X$ is obtained as the proper transform of $X$ in a modification $Y' \to Y$, with $Y'$ nonsingular. 
	\item A \emph{re-embedding principle} assures that the resolution process is independent of the choice of embedding. 
	\item Hironaka's resolution goes by way of a sequence of blow-ups along smooth centers on $Y$. This has the consequence that $Y' \to Y$ as well as $X' \to X$ are projective. But also it implies that the geometry of $Y$ is changed on the way to $Y'$ in a controllable way. 
    	\item Hironaka's resolution is \emph{logarithmic}, in the sense that the exceptional divisor of $Y' \to Y$ is a simple normal crossings divisor meeting $X'$ transversely.
	\item It turns out that Hironaka's resolution is \emph{functorial} with respect to smooth morphisms: if $\widetilde{Y} \to Y$ is smooth and $\widetilde{X} := X \times_Y \widetilde{Y}$, then  Hironaka's resolution for $\widetilde{X} \subset \widetilde{Y}$ is the pull-back of Hironaka's resolution for $X \subset Y$ along $\widetilde{Y} \to Y$.

	\item This in particular assures that it suffices to choose embeddings locally and patch them up. 
\end{enumerate}

This is one of the most important tools at the hands of a birational geometer (and its lack is a serious obstacle for results in positive characteristics).

Hironaka's original proof achieved all this\footnote{To be precise, functoriality is a property explicitly present in later papers of Villamayor \cite{Villamayor}, Schwartz \cite{Schwartz}, Bierstone--Milman \cite{Bierstone-Milman}, and W{\l}odarczyk \cite{Wlodarczyk}.} but falls short of McKernan's criterion of a natural, fully motivated and understandable proof. 
Hironaka, and many following him, worked over the past 60 years to remedy the situation. The present paper is about just one of these efforts, where weighted blow-ups are used.

\subsection{Weighted resolution of singularities}

We continue to work exclusively in characteristic 0.

The work \cite{ATW-weighted} showed how to resolve singularities with a proof satisfying Mckernan's criterion. It has one major drawback --- it does not provide \emph{logarithmic} resolution, a requirement of birational geometry. In particular it cannot on its own serve Mckernan's needs. 

The main result of  \cite{ATW-weighted} provides a resolution of singularities $X' \to X$ where $X'$ is a Deligne--Mumford stack. Some may see this as a drawback, though we do not. In any case, standard techniques of combinatorial nature allow one to replace a resolution using Deligne--Mumford stacks by a resolution using varieties. See  \cite[Proposition 8.1.2]{ATW-weighted}. 

In analogy with Hironaka's proof, the transition in \cite{ATW-weighted}  from $Y$ to $Y'$ involves \emph{weighted blow-ups}, again allowing one to track the change of geometry (e.g the Chow ring, see \cite{Arena-Obinna}), and guaranteeing projectivity on coarse moduli spaces.

\subsection{Toroidal weighted resolution of singularities}

In \cite{Quek}, Quek extended the techniques of  \cite{ATW-weighted} and provided a \emph{toroidal resolution} $X' \to X$, in which $X'$ is a toroidal Deligne--Mumford stack, and the exceptional divisor is subsumed in the toroidal structure. In particular $X'$ may have toroidal singularities. 

Once again, toroidal singularities are easily resolved. In particular Quek deduces a result equivalent in many ways to Hironaka's, \cite[Theorem 1.5]{Quek}. On the other hand, one no longer tightly controls the geometric changes of the ambient variety through $Y' \to Y$, as the transitions involve \emph{toroidal weighted blow-ups}, which are more intricate than weighted blow-ups.

We will recall Quek's methods and intermediate results in Section \ref{Sec:toroidal-resolution} below, as they are relevant to the present exposition.

\subsection{Logarithmic weighted resolution of singularities}

The main result here shows how to combine the results of  \cite{ATW-weighted} with the methods of \cite{Quek} to give a functorial logarithmic resolution of singularties $X'\to X$, adhering to McKernan's principle of a natural, fully motivated and hopefully understandable proof. 
In brief, to a singular point $p$ of a Deligne--Mumford stack $X$ embedded with pure codimension in a smooth Deligne--Mumford stack  $Y$, meeting a simple normal crossings divisor $D$ properly,  we attach an upper-semicontinuous singularity invariant $\loginv^*_X(p)$ taking values in a well-ordered set $\Gamma$. The formation of $\loginv^*$  is functorial with respect to \emph{smooth} morphisms $\widetilde{Y} \to Y$. 
The maximal locus of $\loginv^*$ is the support of a weighted blow-up center $J^*$, which is also functorial. 

 \begin{theorem}[Functorial logarithmic resolution of singularities]\label{Th:log resolution}  The weighted blow-up $Y' \to Y$ of the reduced center associated to $J^*$ is a smooth Deligne--Mumford stack with transformed simple normal crossings divisor $D'$, formed as union of the pre-image of $D$ and the exceptional divisor on $Y'$. The proper transform $X'$ of $X$ satisfies $$\loginv^*_{p'}(X') < \loginv^*_p(X)$$ for any point $p \in X$ and any $p'\in X'$ above it.

 After finitely many iterations, the proper transform $X^{(n)}$ is a smooth locus on a smooth stack $Y^{(n)}$ carrying a simple normal crossings divisor $D^{(n)}$. 
 \end{theorem}

\begin{remark}
See  \cite{Wlodarczyk-cobordant1,Wlodarczyk-cobordant} for an earlier, self-contained approach for the same normal-crossings weighted resolution, which uses a different invariant and blowup center.
Remark \ref{Rem:approaches} briefly indicates the difference of these two approaches. The present paper offers an alternative proof, deriving the result from \cite{ATW-weighted} and \cite{Quek}.
\end{remark}

\begin{remark}
Rather than prove the result from scratch, we deduce it in Section \ref{Sec:general-ideals} from the results of \cite{ATW-weighted}, reviewed in Section \ref{Sec:principalization}, and the results of \cite{Quek}, reviewed in Section \ref{Sec:toroidal-resolution}. In other words, we show that  ``sufficiently strong resolution and  toroidal resolution imply equally strong normal-crossings resolution".

\end{remark}

\subsection{Resolution via embedded resolution and principalization}

Standard techniques for resolution of singularities  reduce the geometric problem to more algebraic ones. 

First, the procedure we devise requires $X$ to be embedded in a smooth variety $Y$. This can always be achieved locally, but to globalize it one needs to verify that the procedure is independent of choices, what we call the \emph{re-embedding principle}, see \cite[Section 8.1]{ATW-weighted}.

Second, instead of working with $X \subset Y$ one works with improving the ideal $\cI_X \subset \cO_Y$. The problem is \emph{principalization} of an ideal $\cI\subset \cO_Y$, which in our case boils down to having the total transform of $\cI$ become exceptional for $Y^{(n)} \to Y$. A simple observation, sometimes known as \emph{accidental resolution}, guarantees that $X$ is resolved along the way, see \cite[Section 6.3]{ATW-weighted}.


\subsection{Where this comes from} In \cite{ABTW-foliated} one proves resolution and principalization in the presence of foliations. In that situation working with extended invariants similar to $\loginv^*$ introduced here becomes essential, and the logarithmic case occurs as a natural byproduct. On the other hand, Quek worked on simplifying his presentation of results of \cite{Quek} for a course at Stanford and a workshop in Heidelberg, and also arrived at the same extended invariant. We decided to combine forces and attempt to describe as natural an argument as we could produce.

\section{Principalization of ideals}\label{Sec:principalization}

The proof of the main result of \cite{ATW-weighted} relies on the notions and ideas we shall first discuss in this section.

\subsection{Derivatives, order, and maximal contact}

\begin{definition} Consider an ideal $\cI \subset \cO_Y$ on a smooth variety $Y$ and the sheaf of differential operators $\cD_Y^{\leq a}$ of order $\leq a$ on $Y$. We define $\cD^{\leq a}_Y(\cI) \subset \cO_Y$ to be the ideal generated by the collection of $\nabla(f)$, with $\nabla $ a local section of  $\cD_Y^{\leq a}$  and $f$ a local section of $\cI$. When $Y$ is clear from context, we write $\cD^{\leq a}$ instead of $\cD_Y^{\leq a}$.
\end{definition}

It is clear that $\cI \subset \cD_Y^{\leq 1}(\cI) \subset\cdots$ is an increasing chain of ideals. If $\cI \neq 0$ it eventually stabilizes at the trivial ideal $(1)$.  

Since we work in characteristic zero, we have that the composition $\cD_Y^{\leq i}\cD_Y^{\leq j} = \cD_Y^{\leq i+j}$, whence $\cD_Y^{\leq i}\cD_Y^{\leq j}(\cI) = \cD_Y^{\leq i+j}(\cI)$. 

Derivative ideals are compatible with localization: $\cD^{\leq a}(\cI_p) = \cD^{\leq a}(\cI)_p.$

\begin{definition} 
We define the \emph{order of $\cI$ at $p$} as follows:  $$\ord_p(\cI) = \min\left\{a \in \ZZ_{\geq 0}: \cD^{\leq a}(\cI_p) = \cO_{Y,p}\right\}.$$ We set $\maxord_Y(\cI) = \max_{p \in Y}\ord_p(\cI)$. The order of the zero ideal is set to $\infty$.
\end{definition}

The locus of order $>a$ is thus the closed vanishing locus of $\cD^{\leq a}(\cI)$, hence order is upper-semicontinuous.

\begin{definition} 
If $\ord_p(\cI)=a$, the \emph{maximal contact ideal} for $\cI$ at $p$ is  $\cD^{\leq a-1}(\cI)_p$. A \emph{maximal contact element} for $\cI$ at $p$ is a section $x \in \cD^{\leq a-1}(\cI)_p$ with $\cD_Y(x)_p = (1)$, that is, some derivative of $x$ is a unit. 
\end{definition}

Since we are working in characteristic 0, maximal contact elements always exist \emph{locally}: if $1 \in \cD^{\leq a}(\cI)$ then there is $\nabla \in \cD^{\leq 1}_{Y,p}$ and $x \in \cD^{\leq a-1}(\cI)_p$ such that $\nabla x = 1$.

Note that the formation of derivative ideals is functorial for smooth morphisms: given a smooth morphism $f \colon \widetilde{Y} \to Y$, we have $\cD_Y^{\leq a}(\cI)\cO_{\widetilde{Y}} = \cD_{\widetilde{Y}}^{\leq a}(\cI \cO_{\widetilde{Y}})$.
It follows that:

\begin{proposition}\label{P:order-and-max-contact-under-smooth-morphisms}
For a smooth morphism $f \colon \widetilde{Y} \to Y$ with $f(\widetilde{p}) = p$, we have $\ord_{\widetilde{p}}(\cI\cO_{\widetilde{Y}}) = \ord_{p}(\cI)$. Moreover, $x\in \cO_{Y,p}$ is a maximal contact element for $\cI$ at $p$ if and only if $f^*x\in \cO_{\widetilde{Y},\widetilde{p}}$ is a maximal contact element for $\cI\cO_{\widetilde{Y}}$ at $\widetilde{p}$. 

If moreover $f$ is surjective, then $\maxord(\cI\cO_{\widetilde{Y}}) = \maxord(\cI)$.
\end{proposition}

We note, however, that the \emph{choice} of a maximal contact $x$ is not unique, and is therefore not functorial.

\subsection{Coefficient ideals}
To go further we need some inductive process, and the standard approach involves induction on dimension by restriction to $V(x)$. Since $x$ only exists locally, functoriality is required to glue the results.

This inductive process requires defining an ideal on $V(x)$ which remembers much of what $\cI$ is. The standard approach involves a coefficient ideal, which is already defined on $Y$.

\begin{definition}[{See \cite[Section 3.5.4]{Kollar}}]
Say $\cI$ has maximal order $a$. We consider $\cD^{\leq a-i}(\cI)$ as having weight $i$, for $i<a$, and take the coefficient ideal $C(\cI,a)$ generated by monomials in $\cD^{\leq a-1}(\cI),\dotsc,\cD^{\leq 1}(\cI),\cI$ of weighted degree $\geq a!$. Concretely,
		$$C(\cI,a) = \sum_{\sum i\cdot b_i \geq a!} \cD^{\leq a-1}(\cI)^{b_1} \
        \cdots \ \cD^{\leq 1}(\cI)^{b_{a-1}}\cdot \cI^{b_a}.$$
\end{definition}

\begin{proposition}
Formation of coefficient ideals is functorial in smooth morphisms: with notation as in Proposition~\ref{P:order-and-max-contact-under-smooth-morphisms}, $C(\cI\cO_{\widetilde{Y}},a) = C(\cI,a) \cO_{\widetilde{Y}}$. 
\end{proposition}

\begin{remark} An alternative approach replaces ideals, coefficient ideals, and centers  by $\QQ$-ideals, equivalently Rees algebras, see \cite{Wlodarczyk-cobordant}. In this approach the coefficient $\QQ$-ideal is an improvement of $\cI$ as an approximation of the final blow-up center, and a step towards its construction. This is reminiscent of ideas proposed in the work \cite{KM16}.   
\end{remark}

\subsection{Invariant and blow-up center}\label{Sec:invariant-and-center}

Assume $\ord_p(\cI)  = a_1$ and  $x_1$ is a maximal contact element for $\cI$ at $p$. If $\cI_p = (x_1^{a_1})$ we are basically done: we define the invariant of the ideal at $p$ to be $\inv_p(\cI) = (a_1)$, and we define the center as $J= (x_1^{a_1})$.
Otherwise we define  $\cI[2] = C(\cI,a)|_{V(x_1)}$, the restriction. By assumption it is not the zero ideal. We may now invoke induction, so that invariants and centers are already defined on $V(x_1)$.

\begin{definition}
Say $\inv_p(\cI[2]) = (b_2,\dotsc, b_k)$, with center $(\bar x_2^{b_2} ,\dotsc , \bar x_k^{b_k})$. Choose \emph{arbitrary} lifts $x_i$ of $\bar x_i$. Remembering that $C(\cI,a_1)$ has weighted degree $a!$ while $\cI$ has weighted degree $a$, we rescale the $b_i$'s by setting $a_i = b_i/(a_1-1)!$ for $i=2,\dotsc,k$. We attach to $\cI$ a local invariant at $p$: $$\inv_p(\cI) = (a_1,a_2,\dotsc,a_k) = \left(a_1,\frac{1}{(a_1-1)!} \cdot \inv_p(C(\cI,a_1)|_{V(x_1)})\right)$$ and a blow-up center around $p$: $$J_{\cI,p} = (x_1^{a_1},x_2^{a_2} ,\dotsc,x_k^{a_k}),$$ a valuative $\QQ$-ideal\footnote{See \cite[Section 2]{ATW-weighted} for definition and properties.} on a neighbourhood of $p$. Note that $a_1 \leq a_2 \leq \dotsm \leq a_k$. We order the set of all possible invariants lexicographically, with the caveat that we consider truncations of such sequences larger, i.e. $(a_1,a_2,\dotsc,a_k) < (a_1,a_2,\dotsc,a_{k-1})$. One can show this set is well-ordered, see \cite[Section 5.1]{ATW-weighted}. We have the key result:

\end{definition}

\begin{proposition}[{See \cite[Theorems 5.1.1 and 5.3.1]{ATW-weighted}}]\label{P:well-definedness-and-functoriality}\
	\begin{enumerate}
		\item The invariant $\inv_p(\cI)$ is independent of choices. 
        \item The center $J_{\cI,p}$ is also independent of choices as valuative $\QQ$-ideal. 
            \item The invariant $\inv_p(\cI)$ is upper semi-continuous on $Y$.
		\item If $\widetilde{Y} \to Y$ is smooth, $\widetilde{\cI} = \cI \cO_{\widetilde{Y}}$, and $\widetilde{p} \mapsto p$, then $$\inv_p(\cI) = \inv_{\widetilde{p}}(\widetilde{\cI})$$ and $$J_{\widetilde{\cI},\widetilde{p}} = J_{\cI,p}\cO_{\widetilde{Y},\widetilde{p}}.$$
	\end{enumerate}
	\end{proposition}

This implies \emph{gluing} of centers, as well as \emph{equivariance} of centers. 

Coefficient ideals possess remarkable properties, formalized by W{\l}o\-darczyk and Koll\'ar \cite{Kollar}, showing they faithfully retain information in $\cI$ yet are more homogeneous. We mention a few properties here:

\begin{proposition}[{See \cite[Corollary 5.3.3]{ATW-weighted}}]\label{P:properties-of-inv-and-centers}
Assume $\ord_p(\cI) = a_1$, with maximal contact $x_1$ for $\cI$ at $p$. Then:
	\begin{enumerate} 
		\item $x_1$ is maximal contact for $C(\cI,a_1)$ at $p$.
		\item $\inv_p(C(\cI,a_1)) = (a_1-1)! \cdot \inv_p(\cI)$, and $\inv_p(\cI^k) = k \cdot \inv_p(\cI)$ for every integer $k \geq 1$.
            \item $J_{C(\cI,a_1),p} = J_{\cI,p}^{(a_1-1)!}$.
\end{enumerate}
\end{proposition}

In part (3) of the proposition, one raises each generator of the valuative $\QQ$-ideal $J_{\cI,p}$ by the power $(a_1-1)!$. That is, if $J_{\cI,p} = (x_1^{a_1},x_2^{a_2} ,\dotsc,x_k^{a_k})$, then $J_{\cI,p}^{(a_1-1)!} = (x_1^{a_1!},x_2^{a_2(a_1-1)!} ,\dotsc,x_k^{a_k(a_1-1)!})$.

\subsection{Smooth weighted blow-ups}\label{Sec:smooth-weighted-blow-ups}

\subsubsection{Global presentation}
Let $\maxinv(\cI) = (a_1,a_2,\dotsc,a_k)$, and let $J = J_\cI$ denote the center obtained by gluing the local centers $J_{\cI,p} = (x_1^{a_1},x_2^{a_2} ,\dotsc , x_k^{a_k})$ around points $p \in Y$ with $\inv_p(\cI) = \maxinv(\cI)$. We recall briefly a presentation of the stack theoretic weighted blow-up associated to $J =(x_1^{a_1},x_2^{a_2} ,\dotsc , x_k^{a_k})$, see \cite{Quek-Rydh, Wlodarczyk-cobordant1}. First, we define 
$$ (a_1,\dotsc, a_k) = \ell (w_1^{-1},\dotsc,w_k^{-1}),$$ where $\ell,w_i$ all integers and $\gcd(w_1,\dotsc,w_k) = 1$, and 
$$\oJ = (x_1^{1/w_1},x_2^{1/w_2} ,\dotsc , x_k^{1/w_k}).$$ Considering $\oJ$ as a valuative $\QQ$-ideal, for any $k\in \ZZ$ we let $\cJ_k = \oJ^k\cap \cO_Y$ and write $$\cA_J := \bigoplus_{k\in \ZZ} \cJ_k,$$ with its ``irrelevant" or ``vertex" ideal $$\cA_{J+}:= \left(\bigoplus_{k>0} \cJ_k\right)$$ generated by terms of positive degree, and morphism $$B_J:= \Spec_Y\cA_J \to \AA^1$$ induced by the grading. Writing $s \in \cI_{-1}\simeq \cO_Y$ for the section corresponding to $1 \in \cO_Y$, the locus $$C_J := V(s) = \Spec_Y \left(\bigoplus_{k\geq 0} \cJ_k/\cJ_{k+1}\right) \subset B_J$$ is the weighted normal cone of $\oJ$, and $B_J$ is the degeneration of $Y$ to the weighted normal cone. Set $B_{J+} := B_J \setminus V(\cA_{\oJ+})$ and $C_{J+} := C_J \setminus V(\cA_{\oJ+})$.

We define the weighted blow-up along $\overline{J}$ to be $$Y'= Bl_\oJ(Y) := [B_{J+} /\GG_m] \to Y$$ with exceptional divisor $[C_{J+}/\GG_m]$.

\subsubsection{Local equations}  A local presentation of this construction is desirable, for which we use an open set where $(x_1,\dotsc,x_k)$ can be completed to local parameters $(x_1,\dotsc,x_n)$ so that $Y \to \AA^n$ is \'etale. We continue to follow \cite{Quek-Rydh}, \cite{Wlodarczyk-cobordant1}.

In this case $$\cA_J =  \cO_Y[s,x_1'\dotsc,x_k']/(x_1-s^{w_1} x_1',\dotsc,x_1-s^{w_k} x_k'),$$ on which $\GG_m$ with parameter $t$ acts via $$t\cdot (s,x_1'\dotsc,x_k') = (t^{-1}s,t^{w_1}x_1'\dotsc,t^{w_k}x_k').$$ In other words, $s$ appears in degree $-1$, and $x_i'$ in degree $w_i$, as the presentation above suggests.

The exceptional locus $V(s)$ in $B_J = \Spec_Y(\cA_J) $ is 
$$\Spec_Y \left(\cO_Y/(x_1,\dotsc,x_k)\left[x_1'\dotsc,x_k'\right]\right),$$ 
which is an $\AA^k$-bundle over the center $V(J) = \Spec_Y \left(\cO_Y/(x_1,\dotsc,x_k)\right)$. We note the resulting presentation of  the vertex ideal $\cA_{\oJ+} = (x_1',\dotsc,x_k')$. Its vanishing locus  is simply $$\Spec_Y \left(\cO_Y/(x_1,\dotsc,x_k)\left[s\right]\right)\quad \simeq \quad V(J) \times \AA^1.$$

\begin{figure}[h]
 \includegraphics[height=1.5in]{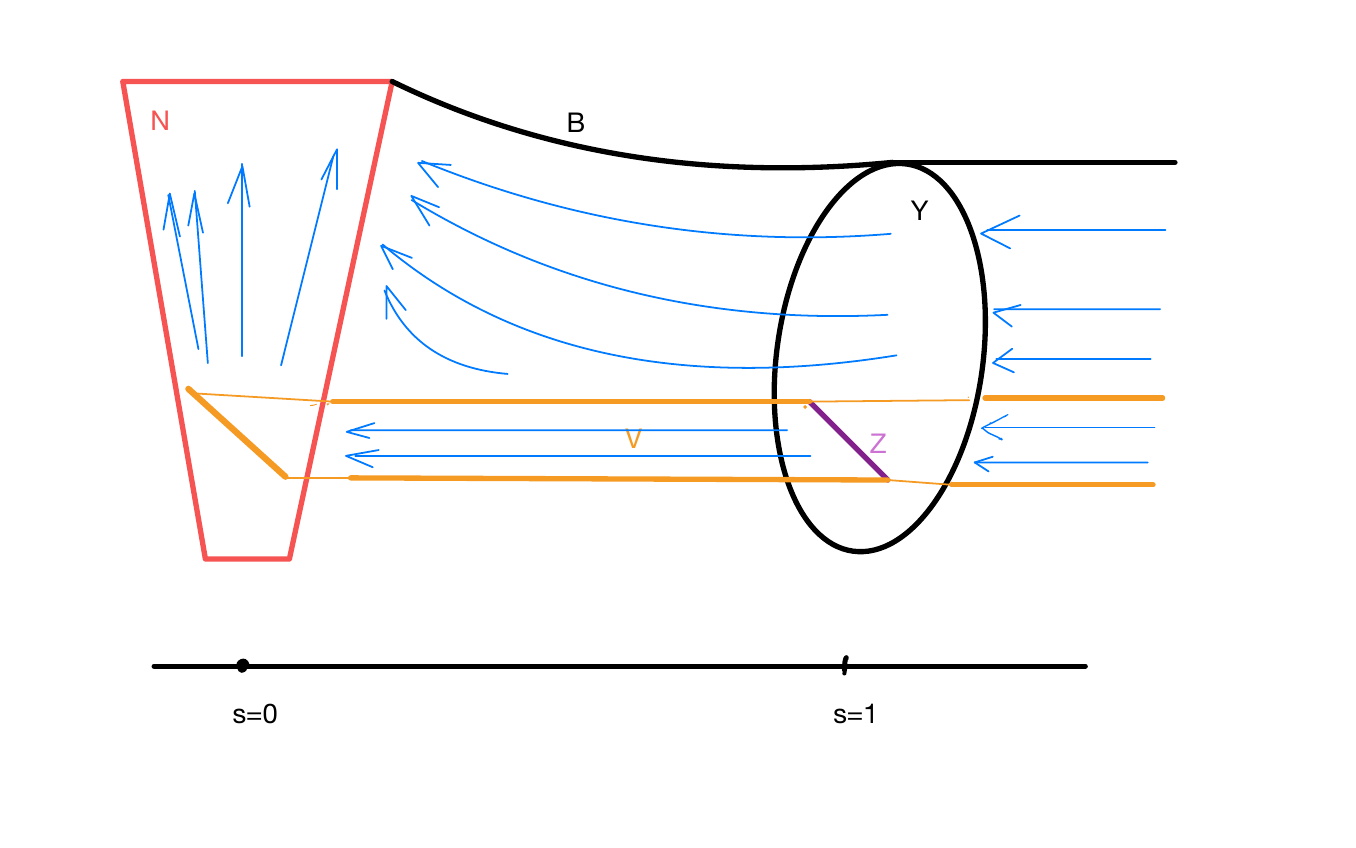} \includegraphics[height=1.5in]{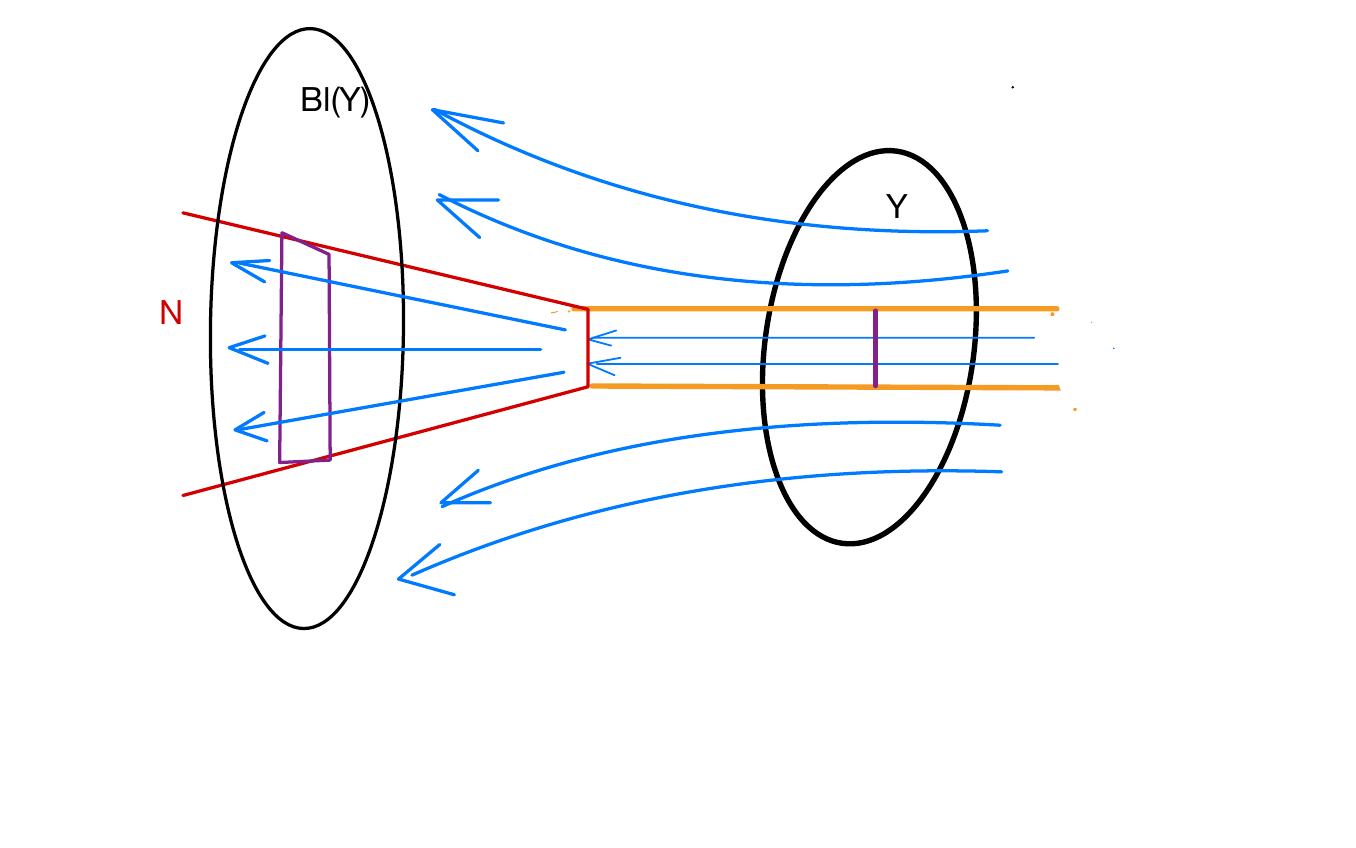}
\caption{Two viewpoints of $B$: on the left, the degeneration to the normal cone of $Y$. On the right, a view from the northeast shows it as a birational cobordism \cite{Wlodarczyk-birational-cobordism}, a link between $Y$ and its blowup. The vertical projection can be viewed as the momentum map $(|s|^2 - \sum w_i |x_i'|^2)/2$, but we digress.}
\end{figure}

\subsection{The invariant drops}\label{Sec:invariant-drops}

The main result of \cite{ATW-weighted} shows the following: 

\begin{theorem}[{See \cite[Theorem 6.2.1]{ATW-weighted}}]\label{Th:invariant drops}
Let $\cI' \subset \cO_{Y'}$ denote the weak transform of $\cI$ under the weighted blow-up $Y' \to Y$ along $\oJ$. Then one has $$\maxinv(\cI') < \maxinv(\cI).$$
After finitely many iterations, one deduces principalization: namely, there is a sequence of smooth weighted blow-ups $Y^{(n)} \to Y^{(n-1)} \to \dotsm \to Y' \to Y$ so that the pullback of $\cI$ to $Y^{(n)}$ is the ideal of the locally principal exceptional divisor on $Y^{(n)}$.
\end{theorem}

We find it useful to re-illuminate this established fact.

We slightly abuse notation, using $\cI'$ for the weak transform of $\cI$ on the degeneration to weighted normal cone $B_J$ as well, i.e. on $B_J$ let us write $\cI\cO_{B_J} =\cI_{C_J}^\ell \cI'$, with $\cI' \not\subset \cI_{C_J}$. This should not cause confusion since it descends to $\cI'$ on $Y'$. In the local presentation above, we have that $\cI\cA_J = (s)^l \cI'$. 

 We note that $B_J \setminus C_J = \GG_m\times Y$, so by functoriality $\maxinv_{B_J \setminus C_J}(\cI') = \maxinv_Y(\cI) = (a_1,\dotsc, a_k)$. Upper semicontinuity implies that $\maxinv_{B_J}(\cI') \geq (a_1,\dotsc, a_k)$. The fundamental reason invariants drop is the following stronger fact:
 
\begin{proposition}\label{Prop:maximality-in-B} We have $$\maxinv_{B_J}(\cI') = \maxinv_Y(\cI) = (a_1,\dotsc, a_k),$$ with maximal locus precisely $V(\cA_{J+}) \subset B_J$.
\end{proposition}
 
 \begin{proof}[Proof of Theorem \ref{Th:invariant drops} given Proposition \ref{Prop:maximality-in-B}]
 
 Since the maximality locus on $B_J$ is $V(A_{J+})$ and $B_{J+} = B_J \setminus V(A_{J+})$, we have $\maxinv_{B_{J+}}(\cI') < \maxinv_{B_J}(\cI')= \maxinv_Y(\cI) = (a_1,\dotsc, a_k).$ Observe that the quotient map $B_{J+} \to Y'$ is, by definition, smooth and surjective. By functoriality we have $\maxinv_{B_{J+}}(\cI') = \maxinv_{Y'}(\cI')$, hence $\maxinv_{Y'}(\cI')<\maxinv_{Y}(\cI)$ as needed.
 \end{proof}
\begin{remark} It is enough  to prove that $\maxinv_{B_J}(\cI') = \maxinv_Y(\cI)$, the maximal locus being $V(\cA_{J+})$ follows by smoothness and functoriality. However all known proofs treat the center at the same time. It is also enough to show that  $\maxinv_{C_J}(\cI') = \maxinv_Y(\cI)$ by upper semicontinuity. See \cite{Wlodarczyk-cobordant} for these statements and their proofs.\end{remark}





To prove the preceding proposition, we require the following:

\begin{lemma}[Giraud,  see e.g. {
\cite[Lemma 4.5]{Encinas-Villamayor-good},
\cite[Lemma 2.6.2]{Wlodarczyk}, 
\cite[Lemma 3.3]{Bierstone-Milman-funct}, 
\cite[Lemma 6.5.2]{ABTW-foliated}
}]\label{Lem:derivatives-vs-weak-transform}
We have 
$$(\cD^{\leq 1}(\cI))' \subset \cD^{\leq 1}(\cI')$$
\end{lemma}
\begin{proof}[Proof of Lemma]


We work this out on $B_J$ using the local presentation above. We note that if $\cI$ is transformed using $(s)^{-\ell}$ then $\cD^{\leq 1}(\cI)$ is transformed using $(s)^{-\ell+w_1}$, where again $w_ia_i = \ell$.  

We prove a slightly stronger result, which will go some way towards results below: we show that $(\cD^{\leq 1}(\cI))' \subset \cD_{B_J/\AA^1}^{\leq 1}(\cI')$. This suffices since $\cD_{B_J/\AA^1}^{\leq 1}(\cI') \subset \cD^{\leq 1}(\cI')$. It also makes computations easier since we need not take $s$-derivarives.

Plugging in $x_i = s^{w_i} x_i'$ we obtain 
$\cI' = (s^{-\ell} f(s^{w_i} x_i') \mid f \in \cI)$, so that, by the chain rule, 
\begin{eqnarray*} \cD_{B/\AA^1}^{\leq 1}(\cI') &= \left(\frac{\partial}{\partial x'_i}(s^{-\ell} f(s^{w_i} x_i')) \mid f\in \cI, i=1,\dotsc,n\right) \\
& = \left(s^{-\ell} \frac{\partial f}{\partial x_i} (s^{w_i} x_i') \cdot s^{w_i} \mid f\in \cI, i=1,\dotsc,n\right)\\
& = \left(s^{-\ell+w_i} \frac{\partial f}{\partial x_i} (s^{w_i} x_i')\mid f\in \cI, i=1,\dotsc,n\right)
\end{eqnarray*}
On the other hand $\cD^{\leq 1}(\cI) = \left(\frac{\partial f}{\partial x_i}(x_i)\mid f\in \cI, i=1,\dotsc,n\right)$ so that $${\cD^{\leq 1}(\cI)' = \left(s^{-\ell+w_1}\frac{\partial f}{\partial x_i}(s^{w_i}x_i')\mid f\in \cI, i=1,\dotsc,n\right)}.$$
Note that
$w_1\geq w_i$ for all $i$, so
		$$\cD^{\leq 1}(\cI)' = \left(s^{-\ell+w_1}\frac{\partial f}{\partial x_i}(s^{w_i}x_i')\right) \subset \left(s^{-\ell+w_i} \frac{\partial f}{\partial x_i} (s^{w_i} x_i'))\right)=\cD_{B/\AA^1}^{\leq 1}(\cI'),$$
as needed.
\end{proof}

\begin{proof}[Proof of Proposition \ref{Prop:maximality-in-B}] 


The lemma implies inductively that $(\cD^{\leq i}(\cI))' \subset \cD^{\leq i}(\cI')$ for every $i \geq 0$. Two particular consequences are that $$1 \in (\cD^{\leq a_1}(\cI))' \subset \cD^{\leq a_1}(\cI') \qquad \text{ and } \qquad 
x_1'\in (\cD^{\leq a_1-1}(\cI))' \subset \cD^{\leq a_1-1}(\cI'),$$ in particular $\maxord(\cI') = a_1$ with maximal contact $x_1'$. Therefore, to prove the proposition, it remains to check $$\max\inv(C(\cI',a_1)|_{V(x_1')}) = \maxinv(C(\cI,a_1)|_{V(x_1)})$$ with maximal locus $V(\cA_{J+}) \cap V(x_1')$.

Next, a third consequence is that $C(\cI,a_1)' \subset C(\cI',a_1)$. Together with the inclusion $\cI'^{(a_1-1)!} \subset C(\cI,a_1)'$, we conclude from Proposition~\ref{P:properties-of-inv-and-centers}(2) that $$\max\inv(C(\cI',a_1)) = \max\inv(C(\cI,a_1)').$$ By Proposition~\ref{P:properties-of-inv-and-centers}(1), $x_1'$ is a maximal contact for $C(\cI',a_1)$. It is also a maximal contact for $C(\cI,a_1)'$ by the same argument at the start of this proof. Therefore, $$\max\inv(C(\cI',a_1)|_{V(x_1')}) = \max\inv(C(\cI,a_1)'|_{V(x_1')}).$$ 

Finally, by induction applied to $V(x_1') \to V(x_1)$ which is the deformation to weighted normal cone of $\overline{J}|_{V(x_1)}$, we have $$\max\inv(C(\cI',a_1)|_{V(x_1')}) = \max\inv\left(\left(C(\cI,a_1)|_{V(x_1)}\right)'\right) =  \maxinv(C(\cI,a_1)|_{V(x_1)})$$ with maximal locus $V(\cA_{J+}) \cap V(x_1') = V(x_2'|_{V(x_1')},\dotsc,x_k'|_{V(x_1')})$, as desired.
\end{proof}

\subsection{Principalization of monomial ideals}

\label{Sec:monomial ideals} 
We revisit an aspect of the above principalization of ideals in the case of monomial ideals (see \cite{Quek, Abramovich-Quek, Wlodarczyk-cobordant}) with an approach which feeds into the later portion of this paper. Assume now that $(Y,D)$ is a pair with $Y$ a variety $D \subset Y$ a simple normal crossings divisor $Y$, and $\cQ \subset \cO_Y$ a non-unit \emph{monomial} ideal on $(Y,D)$, i.e. the stalks of $\cQ$ are generated by monomials in parameters $t_1,\dotsc,t_d$ such that $D = V(t_1t_2\dotsm t_d)$. Such parameters $t_1,\dotsc,t_d$ are called \emph{monomial parameters} on $(Y,D)$.

For a point $p \in Y$, recall that we have the local invariant $\inv_p(\cQ)$ and the local center $J_{\cQ,p}$. If $\maxinv(\cQ) = (b_1,\dotsc, b_m)$ with rational terms $b_1\leq \cdots \leq b_m$, we have a weighted blow-up center $J = J_\cQ$ on $Y$, locally of the form $(y_1^{b_1},\dotsc, y_m^{b_m})$.\footnote{In the earlier sections, the notation $a_i, x_i$ is used instead of $b_i, y_i$, but we replace the notation in anticipation of discussion in Section~\ref{Sec:general-ideals}: $b_i$ would come up in the second stage in our invariant.} 


We note the following well-known  consequence of functoriality, see for instance \cite{BB}:

\begin{proposition}[Principalization of monomial ideals]\label{Prop:principalize-monomial-ideals} The center $J = J_\cQ$ is toroidal, i.e. the parameters $y_1,\dotsc,y_m$ may be taken as monomials, i.e. taken so that $D = V(y_1y_2 \dotsm y_d)$ with $m \leq d$. The same holds for the local centers $J_{\cQ,p}$.

In particular, $Y' \to Y$ is a toroidal morphism with associated simple normal crossings divisor $D' \subset Y'$, and $\cQ'$ is a monomial ideal on $(Y',D')$. After finitely many iterations, there is a sequence of smooth weighted blow-ups $Y^{(n)} \to Y^{(n-1)} \to \dotsm \to Y' \to Y$ so that the pullback of the ideal $\cI$ to $Y^{(n)}$ is a locally principal monomial ideal on the smooth variety $Y^{(n)}$ carrying a normal crossings divisor $D^{(n)}$.
\end{proposition}
\begin{proof} Passing to a local chart we may assume there is an \'etale morphism $\phi:Y \to \AA^n:= \Spec k[t_1, \dotsc, t_n] $ such that $D = \phi^*(V(t_1\cdots t_d))$ for some $1 \leq d \leq n$ and a monomial ideal $\cQ_0\subset \cO_{\AA_n}$ such that $\cQ = \phi^*\cQ_0$. It follows that $\cQ_0$ is invariant under the action of the torus $T :=  \Spec k[t_1^{\pm1}, \dotsc, t_n^{\pm1}]$. Functoriality of centers implies that the center $J_0$ associated to $\cQ_0$ is $T$-invariant, hence its generators may be chosen to be $T$-eigenvectors, namely monomials. Functoriality for \'etale morphisms provides that $J = \phi^*J_0$, which enjoys the same properties, as needed.
\end{proof}

\section{The logarithmic analogues}\label{Sec:toroidal-resolution}

We now replace everything by their logarithmic analogues, first stating Quek's toroidal results in this section. The point is that adding the adjective ``logarithmic" or ``toroidal" at every step works as stated!

\subsection{Logarithmic order, maximal contact, and coefficient ideals}\label{Sec:log-analogues}

Assume $Y$ is provided with a simple normal crossings divisor $D$, giving rise to a logarithmically smooth structure on $Y$ sometimes denoted $(Y|D)$, and sheaves of logarithmic differential operators $\cD^{\leq a}_{(Y|D)}$.

\begin{definition}
We define $\cD^{\leq a}_{(Y|D)}(\cI) \subset \cO_Y$ to be the ideal generated by the collection of $\nabla(f)$, with $\nabla $ a local section of  $\cD_{(Y|D)}^{\leq a}$  and $f$ a local section of $\cI$. This time, the ascending sequence of ideals $$\cI \subset \cD_Y^{\leq 1}(\cI) \subset \cD_Y^{\leq 2}(\cI) \subset \cdots$$ stabilizes at the monomial saturation $\cM_{(Y|D)}(\cI) = \cD^\infty_{(Y|D)}(\cI)$ \cite{Belotto-thesis,Kollar-toroidal,ATW-principalization}, which may or may not be the unit ideal on $Y$.
\end{definition}

\begin{definition}
Given an ideal $\cI \subset \cO_Y$, we define the \emph{logarithmic order of $\cI$ at $p$} as follows: $$\logord_p(\cI) = \min\left\{a: \cD^{\leq a}_{(Y|D)}(\cI)_p = \cO_{Y,p}\right\}$$
with the convention $\logord_p(\cI) = \infty$ if the minimum does not exist, that is, if $\cM_{(Y|D)}(I)_p \neq (1)$. We set $\maxlogord_Y(\cI) = \max_{p \in Y}\logord_p(\cI)$. 
\end{definition}

\begin{definition}
If $\logord_p(\cI)=a<\infty$, the \emph{logarithmic maximal contact ideal} of $\cI$ at $p$ is  $\cD^{\leq a-1}_{(Y|D)}(\cI)_p$. A \emph{logarithmic maximal contact element} for $\cI$ at $p$ is a section   $x\in\cD^{\leq a-1}_{(Y|D)}(\cI)_p$ with $\cD_{(Y|D)}(x)= (1)$, that is, some logarithmic derivative of $x$ is a unit. 
\end{definition}


The formation of logarithmic derivative ideals is functorial with respect to smooth morphisms, and even for logarithmically smooth morphisms. It follows that
for a logarithmically smooth morphism $f \colon \widetilde{Y} \to Y$ with $f(\widetilde{p}) = p$, we have $\logord_{\widetilde{p}}(\cI\cO_{\widetilde{Y}}) = \logord_{p}(\cI)$ and $x\in \cO_{Y,p}$ is a logarithmic maximal contact element if and only if $f^*x\in \cO_{Y_1,p_1}$ is a logarithmic maximal contact element. 



\begin{definition}
Say $\cI$ has maximal logarithmic order $a<\infty$. As before, we consider $\cD^{\leq a-i}_{(Y|D)}(\cI)$ as having weight $i$, for $i<a$, and take the logarithmic coefficient ideal $C_{(Y|D)}(\cI,a)$ generated by all the monomials in $\cD^{\leq a-1}_{(Y|D)}(\cI),\dotsc,\cD(\cI)_{(Y|D)},\cI$ of weighted degree $\geq a!$. Concretely,
$$C_{(Y|D)}(\cI,a) = \sum_{\sum i\cdot b_i \geq a!} \cD^{\leq a-1}_{(Y|D)}(\cI)^{b_1} \; \dotsm \; \cD^{\leq 1}(\cI)^{b_{a-1}}_{(Y|D)}\cdot \cI^{b_a}.$$

The formation of coefficient ideals is functorial with respect to logarithmically smooth morphisms.
\end{definition}

\subsection{Logarithmic invariant and blow-up center}

Recall that one may have $\logord_p(\cI) = \infty$, which happens precisely when $\cM_{(Y|D)}(\cI)_p \neq 1$. In that case, we define $\loginv_p(\cI) = \infty$ and define a logarithmic blow-up center $J^{\log}_{\cI,p} = \cQ := \cM_{(Y|D)}(\cI)$ around $p$. This means that in the induction we must allow logarithmic invariants to be infinite and logarithmic centers to involve monomial ideals and their fractional powers.

Assume $\logord_p(\cI)  = a_1<\infty$ and $x_1$ is a logarithmic maximal contact element. If $\cI_p = (x_1^{a_1})$ we define the logarithmic invariant of the ideal at $p$ to be $\loginv_p(\cI) = (a_1)$, and we define $J^{\log}_{\cI,p} = (x_1^{a_1})$.
Otherwise define  $\cI[2] = C_{(Y|D)}(\cI,a)|_{V(x_1)}$. 

We may again invoke induction, so that logarithmic invariants and centers are already defined on $V(x_1)$. That is, if $\loginv_p(\cI[2]) = (b_2,\dotsc,b_k)$, with center $(\bar x_2^{b_2} ,\dotsc , \bar x_k^{b_k})$, we set $a_i = b_i/(a_1-1)!$ for $i=2,\dotsc,k$, with $$\loginv_p(\cI) = (a_1,\dotsc, a_k)$$ and using arbitrary lifts $x_i$, $$J^{\log}_{\cI,p} = (x_1^{a_1},x_2^{a_2} ,\dotsc , x_k^{a_k}).$$ On the other hand, if $\loginv_p(\cI[2]) = (b_1,\dotsc, b_k,\infty)$, with center $(\bar x_2^{b_2} ,\dotsc , \bar x_k^{b_k},\bar \cQ^{1/e})$, we set $d = (a_1-1)!e$ and we similarly have $\loginv_p(\cI) = (a_1,\dotsc, a_k,\infty)$ with $a_i$ defined as before, and $$J^{\log}_{\cI,p} = (x_1^{a_1},x_2^{a_2},\dotsc,x_k^{a_k},\cQ^{1/d}),$$ where $\cQ$ is the \emph{unique monomial lift} of $\bar \cQ$. Note that $a_1 \leq a_2 \leq \dotsm \leq a_k$. Moreover, from the induction procedure, one can take $d = \prod_{i=1}^k{(a_i-1)!}$.

\subsection{Logarithmic principalization via log smooth weighted blow-ups}

We have the logarithmic analogue of Proposition~\ref{P:well-definedness-and-functoriality}:

\begin{proposition}[{See \cite[Lemma 6.1, Lemma 6.3, Theorem 6.5]{Quek}}]\label{P:well-definedness-and-functoriality-log}\ \begin{enumerate}
    \item The logarithmic invariant $\loginv_p(\cI)$ is independent of choices.
    \item The center $J_{\cI,p}$, as well as the monomial ideal $\cQ$, are independent of choices.
    \item The logarithmic invariant $\loginv_p(\cI)$ is upper semi-continuous on $Y$.
    \item If $\widetilde{Y} \to Y$ is log smooth, $\widetilde{\cI} = \cI\cO_{\widetilde{Y}}$, and $\widetilde{p} \to p$, then $$\loginv_p(\cI) = \loginv_{\widetilde{p}}(\widetilde{\cI})$$ and $$J^{\log}_{\widetilde{\cI},\widetilde{p}} = J^{\log}_{\cI,p}\cO_{\widetilde{Y},\widetilde{p}}.$$
\end{enumerate}
\end{proposition}

Again, this implies \emph{gluing} of centers. That is, it allows one to glue the local logarithmic blow-up centers $J^{\log}_{\cI,p}$ around points $p \in Y$ where $\loginv_p(\cI) = \maxloginv(\cI)$, so as to obtain a center $J^{\log}_\cI$ on $Y$. 

We also have the logarithmic analogue of Proposition~\ref{P:properties-of-inv-and-centers}: 

\begin{proposition}[{See \cite[Corollary 6.8]{Quek}}]\label{P:properties-of-inv-and-centers-log}\
Assume $\ord_p(\cI) = a_1 < \infty$, with logarithmic maximal contact $x_1$ for $\cI$ at $p$. Then:
	\begin{enumerate} 
		\item $x_1$ is logarithmic maximal contact for $C_{(Y|D)}(\cI,a_1)$ at $p$.
		\item $\loginv_p(C_{(Y|D)}(\cI,a_1)) = (a_1-1)! \cdot \loginv_p(\cI)$, where $(a_1-1)! \cdot \infty = \infty$. In addition, $\inv_p(\cI^k) = k \cdot \inv_p(\cI)$ for every integer $k \geq 0$.
            \item $J^{\log}_{C_{(Y|D)}(\cI,a_1),p} = (J^{\log}_{\cI,p})^{(a_1-1)!}$.
	\end{enumerate}
\end{proposition}

Under the logarithmic weighted blow-up $\pi \colon Y' \to Y$ of the reduced center associated to $J^{\log}_\cI$, endowed with logarithmic structure $(Y'|D')$ where $D' \subset Y'$ is given by the union of $\pi^*D$ and the exceptional divisor of $\pi$, it is shown in \cite{Quek} that the weak transform $\cI'$ of $\cI$ satisfies $$\maxloginv(\cI') < \maxloginv(\cI).$$ There is a slight difference here: the new ambient variety $Y'$ is not necessarily smooth, but is at least log smooth when endowed with the logarithmic structure $(Y'|D')$. Therefore, one has to extend the logarithmic invariant $\loginv$ to this setting, namely that of ideals on log smooth varieties. As was already hinted in Proposition~\ref{P:well-definedness-and-functoriality-log}(4), this blow-up $\pi$ satisfies stronger functoriality with respect to \emph{log smooth morphisms} $\widetilde{Y} \to Y$. 

Repeating this procedure, one eventually transforms $Y$ into a log smooth variety, with the total transform of $\cI$ principalized. 


%
%
%
%
%
%
%
%

\section{Logarithmic principalization via smooth weighted blow-ups}\label{Sec:general-ideals}

\subsection{A new extended logarithmic invariant and blow-up center}

Consider again a smooth variety $Y$ with a simple normal crossings divisor $D$. Let $\cI$ be an ideal.  Fix a point $p$ where $\cI$ vanishes.  In the previous section we defined a logarithmic invariant $\loginv_p(\cI)$ and logarithmic center $J^{\log}_{\cI,p}$ associated to $\cI$ at $p$. In this section, we define a slightly different extended logarithmic invariant $\loginv^*_p(\cI)$ and extended logarithmic center $J^{\log,*}_{\cI,p}$ associated to $\cI$ at $p$.
 
The difference between $\loginv$ and $\loginv^*$ arises only in the case where $\loginv_p(\cI) = (a_1,a_2,\dotsc,a_k,\infty)$ with $J^{\log}_{\cI,p} = (x_1^{a_1},x_2^{a_2},\dotsc,x_k^{a_k},\cQ^{1/d})$. In this case, let us consider the invariant $\inv_p(\cQ)=(c_1,\dotsc,c_m)$ and the associated center $J_{\cQ,p} = (y_1^{c_1},\dotsc,y_m^{c_m})$ introduced in Section~\ref{Sec:invariant-and-center}, and let us set the numbers $b_i = c_i/d$. Recall from Proposition~\ref{Prop:principalize-monomial-ideals} that we may and will choose $y_1,\dotsc,y_m$ to be monomial parameters on $(Y,D)$. With these conventions in mind, we define a new extended logarithmic invariant: 
 $$\loginv^*_p(\cI) = \begin{cases} \loginv_p(\cI)=(a_1,a_2,\dotsc,a_k) & \text {if $\infty$ is absent in $\loginv_p(I)$} \\ \\(a_1,\dotsc,a_k,\omega+b_1,\dotsc,\omega+b_m) \quad & 
 \text{\parbox{8.cm}{otherwise, with $\loginv_p(\cI) = (a_1,\dotsc,a_k,\infty)$}} \end{cases}$$ as well as a new extended logarithmic center around $p$: 
 $$J^{\log,*}_{\cI,p} = \begin{cases} J^{\log}_{\cI,p} = (x_1^{a_1},x_2^{a_2},\dotsc,x_k^{a_k}) & \text{if $\infty$ is absent in $\loginv_p(I)$} \\ \\
   (x_1^{a_1},\dotsc,x_k^{a_k},y_1^{b_1},\dotsc,y_m^{b_m})
 \qquad & \text{\parbox{4cm}{otherwise, with $J^{\log}_{\cI,p} = (x_1^{a_1},\dotsc,x_k^{a_k},\cQ^{1/d})$.}} \end{cases}$$
 
 The set of such extended logarithmic invariants is ordered lexicographically similar to before, with the additional caveat that terms $\omega+b_i$ are declared infinitely larger than the rational numbers $a_j$; the notation $\omega$ is here to suggest the first infinite ordinal. 
 Note that we have $a_1 \leq a_2 \leq \dotsm \leq a_k$ and $b_1 \leq b_2 \leq \dotsm \leq b_m$, but each $b_i$ could be smaller or larger than any $a_j$. We have the following analogue of Proposition~\ref{P:properties-of-inv-and-centers-log}:

 \begin{proposition}\label{P:well-definedness-and-functoriality-log*}\
	\begin{enumerate}
	\item The invariant $\loginv_p^*(\cI)$ is independent of choices.
    \item The center $J_{\cI,p}^{\log,*}$ is independent of choices as valuative $\QQ$-ideal. 
    \item The invariant $\loginv_p^*(\cI)$ is upper semi-continuous on $Y$ and takes values in a well-ordered set. 
    \item If $\widetilde{Y} \to Y$ is smooth and log smooth, $\widetilde{\cI} = \cI \cO_{\widetilde{Y}}$, and $\widetilde{p} \mapsto p$, then $$\loginv_p^*(\cI) = \loginv_{\widetilde{p}}^*(\widetilde{\cI})$$ and $$J^{\log,*}_{\widetilde{\cI},\widetilde{p}} = J^{\log,*}_{\cI,p}\cO_{\widetilde{Y},\widetilde{p}}.$$
	\end{enumerate}
	\end{proposition}

\begin{proof}
It suffices to focus on the case where $\loginv_p(\cI) = (a_1,a_2,\dotsc,a_k,\infty)$ with $J^{\log}_{\cI,p} = (x_1^{a_1},x_2^{a_2},\dotsc,x_k^{a_k},\cQ^{1/d})$. Then $\loginv^*_p(\cI)$ is the concatenation of the finite entries $(a_1,\dotsc,a_k)$ in $\loginv_p(\cI)$, followed by $\omega+(1/d) \cdot \inv_p(\cQ)$. Therefore, (1) follows from parts (1) and (2) of Proposition~\ref{P:well-definedness-and-functoriality-log} as well as part (1) of Proposition~\ref{P:well-definedness-and-functoriality}. (3) follows from parts (3) of Propositions~\ref{P:well-definedness-and-functoriality} and~\ref{P:well-definedness-and-functoriality-log}. The assertion in (4) pertaining to $\loginv^*$ follows from parts (4) of Propositions~\ref{P:well-definedness-and-functoriality} and~\ref{P:well-definedness-and-functoriality-log}.

Next, as valuative $\QQ$-ideals, \begin{align*}
    J_{\cI,p}^{\log,*} = (x_1^{a_1},\dotsc,x_k^{a_k},y_1^{b_1},\dotsc,y_m^{b_m}) &= (x_1^{a_1},\dotsc,x_k^{a_k},\cQ^{1/d},y_1^{b_1},\dotsc,y_m^{b_m}) \\
    &= (J_{\cI,p}^{\log},J_{\cQ,p}^{1/d})
\end{align*}
from which (2) follows from part (2) of Proposition~\ref{P:well-definedness-and-functoriality-log}, and the assertion in (4) pertaining to $J^{\log,*}$ follows from parts (4) of Propositions~\ref{P:well-definedness-and-functoriality} and~\ref{P:well-definedness-and-functoriality-log}.
\end{proof}

As before, this allows for \emph{gluing} of centers. Let $J_\cI^{\log,*}$ denote the center on $Y$ obtained by gluing the local centers $J_{\cI,p}^{\log,\ast}$ for $p \in Y$ with $\loginv_p^\ast(\cI) = \maxloginv^\ast(\cI)$.


\subsection{The extended logarithmic invariant drops}

Here is the key result of this paper:

 \begin{theorem}[Functorial logarithmic principalization of ideals] \label{Th:log principalization} Under the weighted blow-up $\pi \colon Y' \to Y$ of the reduced center $\oJ$ associated to $J = J^{\log,*}_\cI$, endowed with the simple normal crossings divisor $D' \subset Y'$ given by the union of $\pi^* D$ and the exceptional divisor $E$ of $\pi$, the weak transform $\cI'$ of $\cI$ satisfies $$\maxloginv^*(\cI') < \maxloginv^*(\cI).$$ After finitely many iterations, the total transform of the ideal $\cI$ is a locally principal monomial ideal on a smooth variety $Y^{(n)}$ carrying a normal crossings divisor $D^{(n)}$. 
 \end{theorem}

As with Proposition~\ref{P:well-definedness-and-functoriality-log*}, it suffices to consider the case where $\maxloginv(\cI) = (a_1,a_2,\dotsc,a_k,\infty)$ with $J^{\log}_{\cI} = (x_1^{a_1},x_2^{a_2},\dotsc,x_k^{a_k},\cQ^{1/d})$. We adopt the same notations and equations as in Section~\ref{Sec:smooth-weighted-blow-ups}, e.g. $\cA_J$, $\cA_{J+}$, $B_J$ and $C_J$. In addition, let $D_J \subset B_J$ denote the simple normal crossings divisor given by the union of $C_J$ together with the pullback of $D \subset Y$ to $B_J$. We endow $B_J$ with the logarithmic structure $(B_J|D_J)$. In the same vein as the proof of Theorem~\ref{Th:invariant drops}, we need to prove the following:

\begin{proposition}\label{Prop:maximalitry-in-B*} We have $$\maxloginv^*_{(B_J|D_J)}(\cI') = \maxloginv^*_{(Y|D)}(\cI) = (a_1,\dotsc, a_k,\omega+b_1,\dotsc,\omega+b_m),$$ with maximal locus precisely $V(\cA_{J+}) \subset B_J$.
\end{proposition}
 
\begin{proof}
We proceed by induction on the number $k$ of finite entries in $\loginv^*_p(\cI)$. If $k=0$, then $\logord_p(\cI) = \infty$, i.e. $J_{\cI,p}^{\log} = \cQ = \cM(I)_p \neq (1)$. Writing $\cI\cO_{B_J} = \cI_{C_J}^\ell \cdot \cI'$ where $\cI'$ is the weak transform of $\cI$ on $B_J$, we have $$\cM_{(B_J|D_J)}(\cI\cO_{B_J}) = \cM_{(B_J|D_J)}(\cI_E^\ell \cdot \cI') = \cI_E^\ell \cdot \cM_{(B_J|D_J)}(\cI').$$ On the other hand, when $k=0$, the morphism $B_J \to Y$ is log smooth, so that \cite[Theorem 3.4.2(3)]{ATW-principalization} implies $$\cM_{(B_J|D_J)}(\cI\cO_{B_J}) = \cM_{(B_J|D_J)}(\cI)\cO_{B_J} = \cQ\cO_{B_J}.$$ Compariung both equations, we deduce that the weak transform $\cQ'$ on $B_J$ under $\pi$ is $\cM_{(B_J|D_J)}(\cI')$. By Proposition~\ref{Prop:maximality-in-B} applied to $\cQ$, we know \begin{align*}\maxinv_{B_J}(\cM_{(B_J|D_J)}(\cI')) = \maxinv_{B_J}(\cQ') &= \maxinv_Y(\cQ) \\ &= \maxinv_{Y}(\cM_{(Y|D)}(\cI)),\end{align*} i.e. $\maxloginv_{(B_J|D_J)}^*(\cI') =\maxloginv_{(Y|D)}^*(\cI)$, as desired.

 If not, $k \geq 1$, i.e. $\logord_p(\cI) = a_1 < \infty$. In this case, one uses the induction hypothesis for $k-1$, similar to the proof of Proposition~\ref{Prop:maximality-in-B} in Section~\ref{Sec:invariant-drops}. We note the required modifications to that proof: \begin{itemize}
     \item Firstly, one needs to prove Lemma~\ref{Lem:derivatives-vs-weak-transform} in this logarithmic setting. The same computation in the proof of that statement actually shows the logarithmic analogue.
     \item One replaces the derivatives and coefficient ideals in that proof by the logarithmic analogues as in Section~\ref{Sec:log-analogues}.
     \item Finally, one replaces the use of Proposition~\ref{P:properties-of-inv-and-centers} by its logarithmic analogue, i.e. Proposition~\ref{P:properties-of-inv-and-centers-log}. \qedhere
 \end{itemize}
 \end{proof} 

 \begin{remark}\label{Rem:approaches} Theorem \ref{Th:log principalization} gives the same final result, of Hironaka-style normal-crossings principalization, as in the earlier \cite{Wlodarczyk-cobordant1,Wlodarczyk-cobordant}. The center $J$ is  different, since those papers bring in the monomial maximal contact parameters $y_i$ in a different way, intertwining with the non-monomial maximal contact parameters rather than separating them. While the method in \cite{Wlodarczyk-cobordant1} is faster and yields a more direct invariant, the maximal contacts in this paper are easier to describe and implement.
The theorem is similar to the main result of \cite{Abramovich-Quek}, which however requires working with Artin stacks.
 \end{remark}

\bibliographystyle{amsalpha}
\bibliography{principalization}

\providecommand{\bysame}{\leavevmode\hbox to3em{\hrulefill}\thinspace}
\providecommand{\MR}{\relax\ifhmode\unskip\space\fi MR }
\providecommand{\MRhref}[2]{%
  \href{http://www.ams.org/mathscinet-getitem?mr=#1}{#2}
}
\providecommand{\href}[2]{#2}
\begin{thebibliography}{AdSTW25}

\bibitem[AdSTW25]{ABTW-foliated}
Dan Abramovich, André~Belotto da~Silva, Michael Temkin, and Jarosław
  Włodarczyk, \emph{Principalization on logarithmically foliated orbifolds},
  2025, \url{https://arxiv.org/abs/2503.00926}.

\bibitem[AOA23]{Arena-Obinna}
Veronica Arena, Stephen Obinna, and Dan Abramovich, \emph{The integral chow
  ring of weighted blow-ups}, 2023, Algebra and Number Theory, to appear,
  \url{https://arxiv.org/abs/2307.01459}.

\bibitem[AQ24]{Abramovich-Quek}
Dan Abramovich and Ming~Hao Quek, \emph{Logarithmic resolution via
  multi-weighted blow-ups}, \'Epijournal G\'eom. Alg\'ebrique \textbf{8}
  (2024), Art. 15, 48. \MR{4838744}

\bibitem[ATW20]{ATW-principalization}
Dan Abramovich, Michael Temkin, and Jaros{\l}aw W{\l}odarczyk,
  \emph{Principalization of ideals on toroidal orbifolds}, J. Eur. Math. Soc.
  (JEMS) \textbf{22} (2020), no.~12, 3805--3866. \MR{4176781}

\bibitem[ATW24]{ATW-weighted}
\bysame, \emph{Functorial embedded resolution via weighted blowings up},
  Algebra Number Theory \textbf{18} (2024), no.~8, 1557--1587 (English).

\bibitem[BdS13]{Belotto-thesis}
Andr\'e~Ricardo Belotto~da Silva, \emph{Resolution of singularities in foliated
  spaces}, Ph.D. thesis, Universit\'e de Haute Alsace, 2013, p.~164.

\bibitem[BdSB19]{BB}
Andr\'e Belotto~da Silva and Edward Bierstone, \emph{Monomialization of a
  quasi-analytic morphism}, July 2019, \url{http://arxiv.org/abs/1907.09502}.

\bibitem[BM97]{Bierstone-Milman}
Edward Bierstone and Pierre~D. Milman, \emph{Canonical desingularization in
  characteristic zero by blowing up the maximum strata of a local invariant},
  Invent. Math. \textbf{128} (1997), no.~2, 207--302. \MR{1440306 (98e:14010)}

\bibitem[BM08]{Bierstone-Milman-funct}
\bysame, \emph{Functoriality in resolution of singularities}, Publ. Res. Inst.
  Math. Sci. \textbf{44} (2008), no.~2, 609--639. \MR{2426359}

\bibitem[EV98]{Encinas-Villamayor-good}
S.~Encinas and O.~Villamayor, \emph{Good points and constructive resolution of
  singularities}, Acta Math. \textbf{181} (1998), no.~1, 109--158. \MR{1654779}

\bibitem[KM16]{KM16}
Hiraku Kawanoue and Kenji Matsuki, \emph{Resolution of singularities of an
  idealistic filtration in dimension 3 after {B}enito-{V}illamayor}, Minimal
  models and extremal rays ({K}yoto, 2011), Adv. Stud. Pure Math., vol.~70,
  Math. Soc. Japan, [Tokyo], 2016, pp.~115--214. \MR{3617780}

\bibitem[Kol07]{Kollar}
J{\'a}nos Koll{\'a}r, \emph{Lectures on resolution of singularities}, Annals of
  Mathematics Studies, vol. 166, Princeton University Press, Princeton, NJ,
  2007. \MR{2289519 (2008f:14026)}

\bibitem[{Kol}19]{Kollar-toroidal}
J{\'a}nos {Koll{\'a}r}, \emph{{Partial resolution by toroidal blow-ups}},
  Tunisian Journal of Mathematics \textbf{1} (2019), no.~1, 3--12.

\bibitem[QR]{Quek-Rydh}
Ming~Hao Quek and David Rydh, \emph{Weighted blow-ups}, In preparation,
  \url{https://people.kth.se/~dary/weighted-blowups20220329.pdf}.

\bibitem[Que22]{Quek}
Ming~Hao Quek, \emph{Logarithmic resolution via weighted toroidal blow-ups},
  Algebr. Geom. \textbf{9} (2022), no.~3, 311--363. \MR{4436684}

\bibitem[Sch92]{Schwartz}
Andrew~Joel Schwartz, \emph{Functorial smoothing of morphisms in equal
  characteristic 0}, Ph.D. thesis, Harvard University, 1992, p.~68.
  \MR{2687722}

\bibitem[Vil89]{Villamayor}
Orlando Villamayor, \emph{Constructiveness of {H}ironaka's resolution}, Ann.
  Sci. \'Ecole Norm. Sup. (4) \textbf{22} (1989), no.~1, 1--32. \MR{985852}

\bibitem[W{\l}o00]{Wlodarczyk-birational-cobordism}
Jaros{\l}aw W{\l}odarczyk, \emph{Birational cobordisms and factorization of
  birational maps}, J. Algebraic Geom. \textbf{9} (2000), no.~3, 425--449.
  \MR{1752010}

\bibitem[W{\l}o05]{Wlodarczyk}
\bysame, \emph{Simple {H}ironaka resolution in characteristic zero}, J. Amer.
  Math. Soc. \textbf{18} (2005), no.~4, 779--822 (electronic). \MR{2163383}

\bibitem[W{\l}o23a]{Wlodarczyk-cobordant1}
Jaros{\l}aw W{\l}odarczyk, \emph{Functorial resolution by torus actions}, 2023,
  \url{arXiv:2203.03090}.

\bibitem[W{\l}o23b]{Wlodarczyk-cobordant}
Jaros{\l}aw W{\l}odarczyk, \emph{Weighted resolution of singularities. {A}
  {R}ees algebra approach}, New techniques in resolution of singularities,
  Oberwolfach Semin., vol.~50, Birkh\"auser/Springer, Cham, 2023, pp.~219--317.
  \MR{4698335}

\end{thebibliography}

\end{document}